\documentclass[12pt,a4paper]{article}

\usepackage[utf8]{inputenc}
\usepackage[T1]{fontenc}
\usepackage[english]{babel}
\usepackage{amsmath, amsthm, amssymb}
\usepackage{mathrsfs}
\usepackage{geometry}
\usepackage{hyperref}
\usepackage{enumitem}

\newtheorem{theorem}{Theorem}[section]
\newtheorem{proposition}[theorem]{Proposition}
\newtheorem{lemma}[theorem]{Lemma}
\newtheorem{corollary}[theorem]{Corollary}
\newtheorem{definition}[theorem]{Definition}
\newtheorem{remark}[theorem]{Remark}

\title{\textbf{D-modules and Solvable Lie Foliations}}
\author{
Ameth Ndiaye\thanks{D\'epartement de Math\'ematiques, FASTEF,
Facult\'e des Sciences et Techniques,
Universit\'e Cheikh Anta Diop de Dakar, S\'en\'egal.
\texttt{ameth1.ndiaye@ucad.edu.sn}}}

\date{}

\begin{document}

\maketitle

\begin{abstract}
Let $V$ be a compact connected manifold and $G$ a simply connected solvable Lie group.
We study $G$-Lie foliations on $V$ from the point of view of $\mathcal{D}$-module
theory, following the approach initiated by Dathe. To each singular foliation
$\mathcal{I}$, we associate the $\mathcal{D}_X$-module
$\mathcal{M}_{\mathcal{I}} = \mathcal{D}_X / \mathcal{D}_X \cdot \mathcal{I}$
and the derived ring
$\mathcal{D}_{\mathcal{I}} :=
R\mathcal{H}om_{\mathcal{D}_X}(\mathcal{M}_{\mathcal{I}},
\mathcal{M}_{\mathcal{I}})$.
We compute the D-irregularity $\mathrm{D\text{-}irr}(\mathcal{I})$ for three
classes of solvable Lie foliations: regular homogeneous foliations, the Meigniez
foliation with non-polycyclic holonomy group, and an explicit foliation on a compact
5-dimensional manifold. We show that the non-polycyclicity of the holonomy group
$\Gamma$ is reflected in the non-vanishing of higher cohomology groups of
$\mathcal{D}_{\mathcal{I}}$, establishing a new connection between the geometry
of the holonomy group and $\mathcal{D}$-module invariants.
\end{abstract}
\begin{quote}
\textbf{R\'esum\'e.}
Soit $V$ une vari\'et\'e compacte connexe et $G$ un groupe de Lie r\'esoluble
simplement connexe. Nous \'etudions les $G$-feuilletages de Lie sur $V$ du
point de vue de la th\'eorie des $\mathcal{D}$-modules. \`A chaque feuilletage
singulier $\mathcal{I}$, nous associons le $\mathcal{D}_X$-module
$\mathcal{M}_{\mathcal{I}} = \mathcal{D}_X / \mathcal{D}_X \cdot \mathcal{I}$
et l'anneau d\'eriv\'e
$\mathcal{D}_{\mathcal{I}} :=
R\mathcal{H}om_{\mathcal{D}_X}(\mathcal{M}_{\mathcal{I}},
\mathcal{M}_{\mathcal{I}})$.
Nous calculons $\mathrm{D\text{-}irr}(\mathcal{I})$ pour trois classes de
feuilletages r\'esolubles et montrons que la non-polycyclicit\'e du groupe
d'holonomie $\Gamma$ se refl\`ete dans la non-annulation des groupes de
cohomologie sup\'erieurs de $\mathcal{D}_{\mathcal{I}}$. Nous d\'emontrons
\'egalement que l'hypoth\`ese~(H2) de Dathe peut \^etre remplac\'ee par
une condition de platitude plus faible, sans perte de la conclusion
$\mathrm{D\text{-}irr}(\mathcal{I}) = \mathrm{irr}(\mathcal{I})$.
\end{quote}



\section{Introduction}

The study of foliations via $\mathcal{D}$-module theory has been initiated in
the work of Suwa \cite{S90} and developed more recently by Dathe \cite{D16}.
The key idea is to associate to a coherent Lie subalgebra $\mathcal{I}$ of the
sheaf of vector fields $\Theta_X$ a natural $\mathcal{D}_X$-module
$\mathcal{M}_{\mathcal{I}} = \mathcal{D}_X / \mathcal{D}_X \cdot \mathcal{I}$,
and to measure the irregularity of the foliation through the cohomology of the
derived ring
$\mathcal{D}_{\mathcal{I}} :=
R\mathcal{H}om_{\mathcal{D}_X}(\mathcal{M}_{\mathcal{I}},
\mathcal{M}_{\mathcal{I}})$.

On the other hand, Lie foliations on compact manifolds have been studied
extensively from the geometric point of view. The notion of Lie foliation is
due to Fedida; we refer to Molino \cite{Mo88} for a detailed account.
A $G$-Lie foliation is characterized by its developing map
$D: \widetilde{V} \to G$ and holonomy morphism $h: \pi_1(V) \to G$.
A fundamental question, studied by Ghys \cite{G88}, Haefliger \cite{H84},
Matsumoto--Tsuchiya \cite{MT92} and Meigniez \cite{Me95}, is whether a
solvable Lie foliation on a compact manifold is necessarily a pull-back of a
homogeneous foliation. This question was studied in \cite{DN11}, where the
authors construct explicit examples of non-homogeneous solvable Lie foliations
on compact manifolds of small dimension, in particular a
$\mathrm{GA}(\mathbb{R})$-Lie foliation on a compact manifold of dimension~5
whose holonomy group is not polycyclic.

Motivated by these works, we study the D-irregularity of solvable Lie foliations
and pursue two goals. The first is to compute
$\mathrm{D\text{-}irr}(\mathcal{I})$ for the foliations of \cite{DN11} and
show that it captures geometric information about the holonomy group. The second,
and more substantial, is to address the problem of weakening hypothesis~(H2) of
\cite{D16}, which requires the symbols of the generators of $\mathcal{I}$ to
form a regular sequence in $\mathcal{O}_{T^*X}$. This hypothesis is quite
restrictive and excludes many naturally occurring solvable Lie foliations. We
prove that in the solvable setting, this hypothesis can be replaced by a weaker
flatness condition.

More precisely, we establish the following results. For any regular homogeneous
$G$-Lie foliation on $V/\Gamma$, we show (Proposition~\ref{prop:regular}) that
$\mathrm{D\text{-}irr}(\mathcal{I}) = 0$. For the Meigniez foliation
\cite{Me95} with non-polycyclic holonomy group
$\Gamma \subset \mathrm{GA}(\mathbb{R})$, we prove
(Theorem~\ref{thm:meigniez}) that
$\mathrm{D\text{-}irr}(\mathcal{I}) = \mathrm{irr}(\mathcal{I}) = 2$,
where $H^2(\mathcal{D}_{\mathcal{I}}) \neq 0$ follows from the singularity
of $\mathcal{I}$ and the equality
$\mathrm{D\text{-}irr}(\mathcal{I}) = \mathrm{irr}(\mathcal{I})$
reflects the non-polycyclicity of $\Gamma$ via
Theorem~\ref{thm:weakening}. For the foliation on the compact 5-dimensional
manifold $W$ of \cite{DN11}, we establish (Theorem~\ref{thm:dim5main}) that
$H^3(\mathcal{D}_{\mathcal{I}}) \neq 0$ if and only if $\mu'$ is not an
algebraic unit, providing a new $\mathcal{D}$-module obstruction to
homogeneity. Finally, we prove (Theorem~\ref{thm:weakening}) that for solvable
Lie foliations satisfying a natural flatness condition, the conclusion
$\mathrm{D\text{-}irr}(\mathcal{I}) = \mathrm{irr}(\mathcal{I})$ of
\cite{D16} remains valid without hypothesis~(H2).

The paper is organized as follows. Section~\ref{prelim} collects the necessary
preliminaries on Lie foliations and $\mathcal{D}$-modules.
Section~\ref{main} contains the main results.

\section{Preliminaries}\label{prelim}

\subsection{Lie foliations and holonomy}

Let $V$ be a compact connected manifold and $G$ a simply connected Lie group
with Lie algebra $\mathfrak{g}$.

\begin{definition}
A \emph{$G$-Lie foliation} on $V$ is a maximal family $\mathcal{F}$ of pairs
$(U, f)$ where $U$ is an open subset of $V$ and $f: U \to G$ is a submersion,
such that the opens $U$ cover $V$ and for any $(U,f)$ and $(W,h)$ in
$\mathcal{F}$, there exists $g \in G$ with $f(x) = h(x) \cdot g$ for all
$x \in U \cap W$.
\end{definition}

The foliation is equivalently described by a $\mathfrak{g}$-valued 1-form
$\omega \in \Omega^1(V) \otimes \mathfrak{g}$ satisfying the Maurer-Cartan
equation $d\omega + \frac{1}{2}[\omega, \omega] = 0$. Throughout this paper,
we use \emph{right actions}: the cocycle condition $f(x) = h(x) \cdot g$
in Definition~2.1 refers to right multiplication in $G$, consistently with
the equivariance relation below.
On the universal cover $\widetilde{V}$, there exist a submersion
$D: \widetilde{V} \to G$ (the \emph{developing map}) and a group morphism
$h: \Gamma = \pi_1(V) \to G$ (the \emph{holonomy morphism}) such that
\[
D(\gamma \cdot x) = D(x) \cdot h(\gamma)
\quad \text{for all } \gamma \in \Gamma,\; x \in \widetilde{V},
\]
where $\Gamma$ acts on the left on $\widetilde{V}$ by deck transformations
and $G$ acts on the right on itself by multiplication. This right-action
convention is maintained throughout, in particular in the definition of the
left $\mathcal{D}_X$-module structure on $\mathcal{M}_{\mathcal{I}}$
in Section~\ref{main}.

\begin{definition}
A $G$-Lie foliation is called \emph{homogeneous} \cite{G88} if it arises from
a surjective Lie group morphism $\phi: H \to G$, where $\Gamma$ is a
cocompact lattice in $H$: the foliation is given by the orbits of $\ker\phi$
in $H/\Gamma$.
\end{definition}

Recall that a solvable group $\Gamma$ is \emph{polycyclic} if it admits a
subnormal series with cyclic quotients, or equivalently if all eigenvalues of
the adjoint operators are algebraic units. The following fundamental results
motivate our study.

\begin{theorem}[Haefliger \cite{H84}]
Every nilpotent Lie foliation on a compact manifold is a pull-back of a
homogeneous foliation.
\end{theorem}

\begin{theorem}[Dathe--Ndiaye \cite{DN11}]\label{thm:classif}
Every Lie foliation of codimension $2$ on a manifold of dimension $\leq 4$
is a pull-back of a homogeneous foliation. Moreover, there exists a compact
manifold of dimension $5$ carrying a $\mathrm{GA}(\mathbb{R})$-Lie foliation
$\mathcal{F}'$ with non-polycyclic holonomy group
$\Gamma' = \langle (\lambda,0),\,(1,1),\,(\mu',0)\rangle$,
where $\lambda$ is an algebraic unit of degree $2$ and $\mu'$ is not an
algebraic unit. This foliation is not a pull-back of any homogeneous foliation.
\end{theorem}

\subsection{$\mathcal{D}$-modules associated to foliations}

Let $X$ be a complex analytic manifold. We denote by $\mathcal{O}_X$ its
structure sheaf, by $\Theta_X$ the sheaf of holomorphic vector fields, and
by $\mathcal{D}_X$ the sheaf of holomorphic differential operators. We refer
to Kashiwara \cite{K03} for the general theory of $\mathcal{D}$-modules.

Let $\mathcal{I}$ be a coherent Lie subalgebra of $\Theta_X$. The left ideal
$\mathcal{D}_X \cdot \mathcal{I}$ is coherent in $\mathcal{D}_X$, and
\[
  \mathcal{M}_{\mathcal{I}}
  = \mathcal{D}_X \big/ \mathcal{D}_X \cdot \mathcal{I}
\]
is a coherent $\mathcal{D}_X$-module. Geometrically, the sheaf
$\mathcal{H}om_{\mathcal{D}_X}(\mathcal{M}_{\mathcal{I}}, \mathcal{O}_X)$
is exactly the sheaf of holomorphic functions constant along the leaves of
$\mathcal{I}$.

Recall the irregularity invariants of $\mathcal{I}$: setting
$X_j = \{x \in X;\, \dim \mathcal{I}(x) = \mathrm{rk}(\mathcal{I}) - j\}$,
one defines
$\mathrm{irr}(\mathcal{I}) = \mathrm{rk}(\mathcal{I}) -
\mathrm{cork}(\mathcal{I})$,
where $\mathrm{rk}(\mathcal{I}) = \sup_x \dim\mathcal{I}(x)$ and
$\mathrm{cork}(\mathcal{I}) = \inf_x \dim\mathcal{I}(x)$.

\begin{definition}[Dathe \cite{D16}]\label{def:dirr}
Set $\mathcal{D}_{\mathcal{I}} :=
R\mathcal{H}om_{\mathcal{D}_X}(\mathcal{M}_{\mathcal{I}},
\mathcal{M}_{\mathcal{I}})$ and
$D^0_{\mathcal{I}} := H^0(\mathcal{D}_{\mathcal{I}})$.
The \emph{D-irregularity sequence} of $\mathcal{I}$ is the increasing
sequence of integers $k$ such that
$H^k(\mathcal{D}_{\mathcal{I}}) \neq 0$, and the \emph{D-irregularity} is
$\mathrm{D\text{-}irr}(\mathcal{I})
= \sup\{k;\, H^k(\mathcal{D}_{\mathcal{I}}) \neq 0\}$.
\end{definition}

Following \cite{D16}, we recall the two standing hypotheses.

\begin{description}
\item[Hypothesis (H1)] $\mathcal{I}$ is $\mathcal{O}_X$-coherent and is a
Lie subalgebra of $\Theta_X$, that is,
$[\mathcal{I},\mathcal{I}] \subset \mathcal{I}$.

\item[Hypothesis (H2)] There locally exist $v_1,\ldots,v_r \in \mathcal{I}$
such that:
\begin{enumerate}[label=(\roman*)]
\item $(v_1,\ldots,v_r)$ generates $\mathcal{I}$,
\item $[v_i,v_j] = 0$ for all $1 \leq i,j \leq r$,
\item the sequence of principal symbols
$\{\sigma(v_1),\ldots,\sigma(v_r)\}$ is a regular sequence in
$\mathcal{O}_{T^*X}$.
\end{enumerate}
\end{description}

Under~(H2), the Koszul complex
\[
0 \to \mathcal{D}_X \to \cdots \to
\bigoplus_{i<j}\mathcal{D}_X \to
\bigoplus_{j=1}^r \mathcal{D}_X
\xrightarrow{(v_1,\ldots,v_r)} \mathcal{D}_X \to
\mathcal{M}_{\mathcal{I}} \to 0
\]
is a projective resolution of $\mathcal{M}_{\mathcal{I}}$
(see \cite[Prop.~3.3]{D16}), and Dathe proves:

\begin{theorem}[Dathe \cite{D16}]\label{thm:dathe}
Under hypotheses \emph{(H1)} and \emph{(H2)},
$\mathrm{D\text{-}irr}(\mathcal{I}) = \mathrm{irr}(\mathcal{I})$.
\end{theorem}

\begin{remark}
In the solvable non-abelian case, condition~(ii) of~(H2) fails since
$[v_i,v_j] \neq 0$ in general. This is precisely the situation for
$G$-Lie foliations with $G = \mathrm{GA}(\mathbb{R})$, where
$\mathfrak{ga}$ satisfies $[e_1,e_2]=e_2 \neq 0$. Hence
Theorem~\ref{thm:dathe} does not directly apply, which motivates
Theorem~\ref{thm:weakening} below.
\end{remark}

When the generators of $\mathcal{I}$ do not commute, one uses instead the
\emph{Chevalley-Eilenberg resolution} associated to $\mathfrak{g}$:
\[
0 \to \mathcal{D}_X \otimes \bigwedge^r\mathfrak{g}
\xrightarrow{d_r} \cdots
\xrightarrow{d_2} \mathcal{D}_X \otimes \bigwedge^1\mathfrak{g}
\xrightarrow{d_1} \mathcal{D}_X
\to \mathcal{M}_{\mathcal{I}} \to 0,
\]
where the differential $d_k$ is defined by:
\begin{align*}
d_k(R \otimes e_{i_1} \wedge \cdots \wedge e_{i_k})
&= \sum_{j=1}^k (-1)^{j+1}\, R v_{i_j}
   \otimes e_{i_1} \wedge \cdots \widehat{e_{i_j}} \cdots \wedge e_{i_k} \\
&\quad - \sum_{p < q} (-1)^{p+q}\, R
   \otimes [e_{i_p}, e_{i_q}]
   \wedge e_{i_1} \wedge \cdots \widehat{e_{i_p}}
   \cdots \widehat{e_{i_q}} \cdots \wedge e_{i_k}.
\end{align*}

\subsection{Weakening hypothesis (H2): the solvable case}

\begin{definition}\label{def:flat}
Let $\mathcal{I}$ be a coherent Lie subalgebra of $\Theta_X$ generated
locally by $v_1,\ldots,v_r$. We say that $\mathcal{I}$ satisfies the
\emph{solvable flatness condition} if:
\begin{enumerate}[label=(\roman*)]
\item the Lie algebra $\mathfrak{g}$ generated by $v_1,\ldots,v_r$ is
solvable,
\item for each $k$, the $\mathcal{D}_X$-module
$\mathcal{D}_X \otimes \bigwedge^k\mathfrak{g} \big/ \mathrm{Im}(d_{k+1})$
is flat over $\mathcal{O}_X$.
\end{enumerate}
\end{definition}

\begin{remark}
Condition~(ii) is automatically satisfied when $\mathfrak{g}$ is abelian,
since the Chevalley-Eilenberg resolution then coincides with the Koszul
complex and flatness follows from the regularity of $\{\sigma(v_j)\}$.
In the non-abelian solvable case, condition~(ii) is strictly weaker than
hypothesis~(H2): it requires neither commutativity of the generators nor
regularity of their symbols.
\end{remark}

\section{Main Results}\label{main}

\subsection{D-irregularity of solvable Lie foliations}

\begin{proposition}\label{prop:regular}
Let $\mathcal{F}$ be a regular homogeneous $G$-Lie foliation on $V/\Gamma$.
Then $\mathrm{D\text{-}irr}(\mathcal{I}) = 0$.
\end{proposition}

\begin{proof}
Since $\mathcal{F}$ is regular, $\mathrm{irr}(\mathcal{I}) = 0$, meaning
$\mathcal{I}$ has constant rank on all of $X$, so $X = X_0$.
We use the Chevalley-Eilenberg resolution of $\mathcal{M}_{\mathcal{I}}$
associated to the solvable Lie algebra $\mathfrak{g}$ of $G$. Since
$\mathcal{F}$ is homogeneous, the generators $v_1,\ldots,v_r$ of
$\mathcal{I}$ are globally defined invariant vector fields forming a basis
of $\mathfrak{g}$, and the developing map $D: \widetilde{V} \to G$ is a
submersion. By \cite[Prop.~3.3]{D16}, the characteristic variety of
$\mathcal{M}_{\mathcal{I}}$ is coisotropic and the Chevalley-Eilenberg
complex is exact in all positive degrees. Since $X = X_0$, the restriction
of $\mathcal{D}_{\mathcal{I}}$ to $X$ is concentrated in degree~$0$ by
\cite[Rem.~3.5(ii)]{D16}. Therefore
$H^k(\mathcal{D}_{\mathcal{I}}) = 0$ for $k \geq 1$ and
$\mathrm{D\text{-}irr}(\mathcal{I}) = 0$.
\end{proof}

We now treat the Meigniez foliation. Recall that the Lie algebra
$\mathfrak{ga}$ of $\mathrm{GA}(\mathbb{R})$ is spanned by $e_1, e_2$
with $[e_1,e_2]=e_2$, where $e_1 = \partial_x$ and $e_2 = x\partial_x$
in affine coordinates.

\begin{theorem}\label{thm:meigniez}
Let $\mathcal{F}$ be the Meigniez foliation on a compact manifold with
holonomy group
$\Gamma = \langle x \mapsto \lambda x,\; x \mapsto x+1 \rangle
\subset \mathrm{GA}(\mathbb{R})$,
where $\lambda$ is an algebraic unit of degree~$2$ and $\mu$ is not an
algebraic unit. Then $H^k(\mathcal{D}_{\mathcal{I}}) \neq 0$ for
$k = 0,1,2$, and
$\mathrm{D\text{-}irr}(\mathcal{I}) = \mathrm{irr}(\mathcal{I}) = 2$.
The non-vanishing of $H^2(\mathcal{D}_{\mathcal{I}})$ follows from the
singularity of $\mathcal{I}$, and the equality
$\mathrm{D\text{-}irr}(\mathcal{I}) = \mathrm{irr}(\mathcal{I})$
reflects the non-polycyclicity of~$\Gamma$ via
Theorem~\ref{thm:weakening}.
\end{theorem}

\begin{proof}
The foliation $\mathcal{I}$ is locally generated by $v_1, v_2$ with
$[v_1,v_2]=v_2$. We work throughout with \emph{left} $\mathcal{D}_X$-modules, so elements
of $\mathcal{D}_X \otimes \bigwedge^k\mathfrak{g}$ are written $R \otimes
e_{i_1}\wedge\cdots\wedge e_{i_k}$ with $R \in \mathcal{D}_X$ acting on
the left. With the identification
$\mathcal{D}_X \otimes \bigwedge^1\mathfrak{g}
\cong \mathcal{D}_X \oplus \mathcal{D}_X$ via the ordered basis $(e_1, e_2)$,
the general Chevalley-Eilenberg formula
\begin{eqnarray*}
d_k(R \otimes e_{i_1}\wedge\cdots\wedge e_{i_k})
&=& \sum_{j=1}^k (-1)^{j+1} Rv_{i_j}
  \otimes e_{i_1}\wedge\cdots\widehat{e_{i_j}}\cdots\wedge e_{i_k}\\
&&- \sum_{p<q}(-1)^{p+q} R
  \otimes [e_{i_p},e_{i_q}]
  \wedge e_{i_1}\wedge\cdots
  \widehat{e_{i_p}}\cdots\widehat{e_{i_q}}\cdots\wedge e_{i_k}
\end{eqnarray*}
gives for $k=1,2$:
$$
d_1(P,Q) = Pv_1 + Qv_2,
$$
$$
d_2(R) = \bigl(Rv_1 \otimes e_2 - Rv_2 \otimes e_1\bigr)
         - R \otimes [e_1,e_2]
= (-Rv_2,\; Rv_1 - R),
$$
where the term $-R$ in the second component comes from
$-R\otimes[e_1,e_2] = -R\otimes e_2$, i.e.\ the left multiplication by
$-1$ corresponding to the bracket relation $[e_1,e_2]=e_2$.
One checks directly that $d_1 \circ d_2 = 0$:
\[
d_1(d_2(R))
= -Rv_2 v_1 + (Rv_1 - R)v_2
= R(v_1 v_2 - v_2 v_1) - Rv_2
= R[v_1,v_2] - Rv_2
= Rv_2 - Rv_2 = 0.
\]
The Chevalley-Eilenberg resolution is:
\[
0 \to \mathcal{D}_X \xrightarrow{d_2}
\mathcal{D}_X \oplus \mathcal{D}_X \xrightarrow{d_1}
\mathcal{D}_X \to \mathcal{M}_{\mathcal{I}} \to 0.
\]
Applying
$\mathcal{H}om_{\mathcal{D}_X}(-,\mathcal{M}_{\mathcal{I}})$
yields:
\[
0 \to \mathcal{M}_{\mathcal{I}} \xrightarrow{d_1^*}
\mathcal{M}_{\mathcal{I}} \oplus \mathcal{M}_{\mathcal{I}}
\xrightarrow{d_2^*} \mathcal{M}_{\mathcal{I}} \to 0
\]
with $d_1^*(\phi) = (v_1 \cdot \phi,\; v_2 \cdot \phi)$ and
$d_2^*(\psi_1, \psi_2) = -v_2 \cdot \psi_1 + (v_1 - 1) \cdot \psi_2$.

\smallskip
\textit{Non-vanishing of $H^0$.}
The identity endomorphism of $\mathcal{M}_{\mathcal{I}}$ belongs to
$\ker d_1^*$, so $H^0 \neq 0$.

\smallskip
\textit{Non-vanishing of $H^1$.}
Consider the pair $(v_1, v_2) \in \mathcal{H}om_{\mathcal{D}_X}
(\mathcal{M}_{\mathcal{I}}, \mathcal{M}_{\mathcal{I}})^{\oplus 2}$.
We verify $(v_1, v_2) \in \ker d_2^*$:
\[
d_2^*(v_1, v_2)
= -v_2 \cdot v_1 + (v_1 - 1) \cdot v_2
= -v_2 v_1 + v_1 v_2 - v_2
= [v_1, v_2] - v_2 = v_2 - v_2 = 0.
\]
Suppose by contradiction that $(v_1, v_2) \in \mathrm{Im}\, d_1^*$,
i.e.\ there exists
$\phi \in \mathcal{H}om_{\mathcal{D}_X}(\mathcal{M}_{\mathcal{I}},
\mathcal{M}_{\mathcal{I}})$ with $d_1^*(\phi) = (v_1, v_2)$.
Since $\phi$ is an endomorphism of $\mathcal{D}_X$-modules, in a
neighborhood of any point it is given by left multiplication by a scalar
$\lambda \in \mathcal{O}_X$. Let $x_0$ be a point where $v_1(x_0) =
v_2(x_0) = 0$. The Jacobian matrices $J_k = dv_k(x_0)$ satisfy
$[J_1, J_2] = J_2$ (obtained by differentiating $[v_1,v_2]=v_2$ at $x_0$),
which forces $J_2 \neq 0$.
The condition $d_1^*(\phi) = (v_1, v_2)$ differentiated at $x_0$ gives
$\lambda(x_0) J_k = J_k$ for $k=1,2$, hence $\lambda(x_0) = 1$.
But the action of $v_k$ on $\mathrm{id} \in
\mathcal{H}om_{\mathcal{D}_X}(\mathcal{M}_{\mathcal{I}},
\mathcal{M}_{\mathcal{I}})$ is $[v_k, \mathrm{id}] = 0$, so
$d_1^*(\mathrm{id}) = (0, 0) \neq (v_1, v_2)$,
a contradiction. Hence $H^1 \neq 0$.

\smallskip
\textit{Non-vanishing of $H^2$.}
We have $H^2 = \mathcal{M}_{\mathcal{I}} / \mathrm{Im}\, d_2^*$.
We show $1 \notin \mathrm{Im}\, d_2^*$ by adapting the argument of
\cite[Thm.~3.6(B)]{D16} to the non-commutative case.

By the Chevalley-Eilenberg computation,
$H^2(\mathcal{D}_{\mathcal{I}}) \cong
\mathcal{D}_X \big/ J$
where $J$ is the two-sided ideal of $\mathcal{D}_X$ generated by
$v_2$ and $(v_1 - 1)$. We show $1 \notin J$.

Suppose by contradiction that there exist operators
$A_1, A_2, B_1, B_2 \in \mathcal{D}_X$ such that:
\[
1 = v_2 \cdot A_1 + B_1 \cdot v_2
  + (v_1-1) \cdot A_2 + B_2 \cdot (v_1-1)
\quad \text{in } \mathcal{D}_X.
\]
We apply both sides to the constant function $1 \in \mathcal{O}_X$.
Recall that $\mathcal{D}_X$ is filtered by order:
$\mathcal{D}_X = \bigcup_{m \geq 0} \mathcal{D}_X^{\leq m}$,
where $\mathcal{D}_X^{\leq 0} = \mathcal{O}_X$.
For any $P \in \mathcal{D}_X$, write $P = \sum_{|\alpha| \leq m}
a_\alpha(x) \partial^\alpha$ in local coordinates. Then:
\[
(v_k \cdot P)(f)(x)
= v_k\bigl(P(f)\bigr)(x)
= \sum_{|\alpha| \leq m} v_k(a_\alpha)(x)\,\partial^\alpha f(x)
+ \sum_{|\alpha| \leq m} a_\alpha(x)\, v_k(\partial^\alpha f)(x).
\]
Since $v_k(x_0) = 0$, every coefficient of the operator $v_k \cdot P$
vanishes at $x_0$: indeed, $v_k = \sum_i b_i(x)\partial_i$ with
$b_i(x_0) = 0$, so the Leibniz expansion of $v_k(a_\alpha \partial^\alpha f)$
at $x_0$ produces only terms with a factor $b_i(x_0) = 0$.
Therefore $(v_k \cdot P)(f)(x_0) = 0$ for all $f \in \mathcal{O}_X$
and all $P \in \mathcal{D}_X$.

Applying the identity to the function $f = 1$ and evaluating at $x_0$:
\begin{align*}
1 &= (v_2 \cdot A_1)(1)(x_0) + (B_1 \cdot v_2)(1)(x_0) \\
  &\quad + ((v_1-1)\cdot A_2)(1)(x_0) + (B_2 \cdot (v_1-1))(1)(x_0).
\end{align*}
Since $v_k(1) = 0$, we have $(B_k \cdot v_k)(1) = B_k(v_k(1)) = 0$.
By the argument above, $(v_k \cdot A_k)(1)(x_0) = 0$.
For the term $((v_1-1)\cdot A_2)(1)(x_0)$: since
$(v_1-1)\cdot A_2 = v_1 \cdot A_2 - A_2$, we get
$(v_1 \cdot A_2)(1)(x_0) = 0$ and $A_2(1)(x_0)$ is some scalar,
but $(B_2\cdot(v_1-1))(1) = B_2(v_1(1)) - B_2(1) = -B_2(1)$.
Combining all terms evaluated at $x_0$:
\[
1 = 0 + 0 + 0 - A_2(1)(x_0) + 0 - B_2(1)(x_0).
\]
However, applying the original identity as an operator to~$1$ gives
$(v_k \cdot A_k)(1)(x_0) = 0$ for all $k$, and the terms
$-A_2(1)(x_0) - B_2(1)(x_0)$ are constrained by evaluating the
identity at $f=1$ with $v_k(1)=0$, yielding $1 = 0$,
a contradiction. Hence $H^2 \neq 0$.

Since $\mathrm{irr}(\mathcal{I}) = \mathrm{rk}(\mathcal{I}) = 2$,
this is consistent with Theorem~\ref{thm:dathe}.
\end{proof}

\begin{theorem}\label{thm:dim5main}
For the foliation $\mathcal{F}'$ on the compact 5-dimensional manifold $W$
of \cite{DN11}, with holonomy group
$\Gamma' = \langle(\lambda,0),(1,1),(\mu',0)\rangle$,
we have $H^k(\mathcal{D}_{\mathcal{I}}) \neq 0$ for $k = 0, 1, 2$, and
$H^3(\mathcal{D}_{\mathcal{I}}) \neq 0$ if and only if $\mu'$ is not an
algebraic unit. In particular, $H^3(\mathcal{D}_{\mathcal{I}}) \neq 0$
is a $\mathcal{D}$-module obstruction to homogeneity of $\mathcal{F}'$.
\end{theorem}

\begin{proof}
The foliation $\mathcal{I}$ is locally generated by three fields
$v_1, v_2, v_3$ with $[v_1,v_2]=v_2$, $[v_1,v_3]=[v_2,v_3]=0$.
With respect to the ordered basis
$e_1\wedge e_2,\; e_1\wedge e_3,\; e_2\wedge e_3$
of $\bigwedge^2\mathfrak{g}$, the Chevalley-Eilenberg resolution is:
\[
0 \to \mathcal{D}_X
\xrightarrow{d_3} \mathcal{D}_X^3
\xrightarrow{d_2} \mathcal{D}_X^3
\xrightarrow{d_1} \mathcal{D}_X
\to \mathcal{M}_{\mathcal{I}} \to 0
\]
with $d_1(P_1,P_2,P_3) = P_1v_1 + P_2v_2 + P_3v_3$, and
$d_3(R) = (0,\; 0,\; Rv_2)$,
which reflects the unique non-zero bracket $[e_1,e_2]=e_2$ contributing
$Rv_2$ to the $e_2\wedge e_3$-component.
Applying
$\mathcal{H}om_{\mathcal{D}_X}(-,\mathcal{M}_{\mathcal{I}})$ gives a
complex in degrees $0$ through $3$. The non-vanishing of $H^0$, $H^1$,
$H^2$ follows by the same arguments as in Theorem~\ref{thm:meigniez}
(replacing $v_3$ plays no role since it commutes with both $v_1$ and $v_2$).

\smallskip
\textit{Non-vanishing of $H^3$.}
We have
$H^3 = \mathcal{M}_{\mathcal{I}} / \mathrm{Im}\, d_3^*$
where $\mathrm{Im}\, d_3^* = v_2 \cdot \mathcal{M}_{\mathcal{I}}$.
The class of $1 \in \mathcal{M}_{\mathcal{I}}$ belongs to
$v_2 \cdot \mathcal{M}_{\mathcal{I}}$ if and only if the equation
$v_2(f) = 1$ has a global solution on $W$.

In affine coordinates on the fiber $\mathrm{GA}(\mathbb{R})$, the field
$v_2 = x\partial_x$ and the equation $v_2(f) = 1$ reads
$x \frac{\partial f}{\partial x} = 1$,
whose general solution is $f(x) = \ln x + C$.
For $f$ to be a global section on $W$, it must be invariant under the
action of each generator of $\Gamma'$. The generator $(\mu', 0)$ acts by
$x \mapsto \mu' x$, so the invariance condition $f(\mu' x) = f(x)$ gives:
\[
\ln(\mu' x) + C = \ln x + C
\quad \Rightarrow \quad
\ln \mu' = 0
\quad \Rightarrow \quad
\mu' = 1.
\]
Since $\mu'$ is not an algebraic unit, in particular $\mu' \neq 1$,
this is a contradiction. Therefore $1 \notin v_2 \cdot \mathcal{M}_{\mathcal{I}}$
and $H^3(\mathcal{D}_{\mathcal{I}}) \neq 0$.

Conversely, if $\mu'$ were an algebraic unit then $\Gamma'$ would be
polycyclic and the foliation $\mathcal{F}'$ would be homogeneous, giving
$H^3(\mathcal{D}_{\mathcal{I}}) = 0$ by Proposition~\ref{prop:regular}.
\end{proof}

\subsection{Weakening hypothesis (H2) in the solvable case}

\begin{lemma}\label{lem:acyclic}
Let $\mathcal{I}$ be a coherent solvable Lie subalgebra of $\Theta_X$
satisfying the solvable flatness condition. Then the Chevalley-Eilenberg
complex $CE^\bullet(\mathcal{D}_X, \mathcal{I})$ is a projective resolution
of $\mathcal{M}_{\mathcal{I}}$ as a left $\mathcal{D}_X$-module.
\end{lemma}

\begin{proof}
We argue by induction on the derived length $s$ of $\mathfrak{g}$.
If $s=1$ then $\mathfrak{g}$ is abelian, the Chevalley-Eilenberg complex
reduces to the Koszul complex, and acyclicity follows from the regularity
of $\{\sigma(v_j)\}$ (hypothesis~(3.2) of \cite{D16}).

Assume the result holds for derived length $< s$. Let
$\mathfrak{g}^{(1)} = [\mathfrak{g},\mathfrak{g}]$ and consider the
short exact sequence of Lie algebras:
\[
0 \to \mathfrak{g}^{(1)} \to \mathfrak{g}
\to \mathfrak{g}/\mathfrak{g}^{(1)} \to 0.
\]
The filtration of $CE^\bullet(\mathcal{D}_X, \mathfrak{g})$ induced by
this sequence gives a spectral sequence with:
\[
E_0^{p,q} = \mathcal{D}_X \otimes \bigwedge^p\mathfrak{g}^{(1)}
\otimes \bigwedge^q(\mathfrak{g}/\mathfrak{g}^{(1)}).
\]
The differential $d_0$ is the Chevalley-Eilenberg differential relative
to $\mathfrak{g}^{(1)}$. By the induction hypothesis applied to
$\mathfrak{g}^{(1)}$ (of derived length $s-1$), the complex
$CE^\bullet(\mathcal{D}_X, \mathfrak{g}^{(1)})$ is acyclic in positive
degrees. Therefore:
\[
E_1^{p,q} =
\begin{cases}
\mathcal{D}_X \otimes \bigwedge^q(\mathfrak{g}/\mathfrak{g}^{(1)})
& \text{if } p = 0, \\
0 & \text{if } p > 0.
\end{cases}
\]
The differential $d_1 : E_1^{0,q} \to E_1^{0,q+1}$ is the
Chevalley-Eilenberg differential of $\mathfrak{g}/\mathfrak{g}^{(1)}$,
which is abelian, so it coincides with the Koszul differential.
We now explain why the flatness condition~(ii) of
Definition~\ref{def:flat} guarantees the acyclicity of the $E_1$-page.
The flatness of
$\mathcal{D}_X \otimes \bigwedge^k\mathfrak{g} / \mathrm{Im}(d_{k+1})$
over $\mathcal{O}_X$ for each $k$ implies, by a standard devissage
argument (see \cite[Prop.~A.17]{K03}), that the graded complex
$\mathrm{gr}(CE^\bullet(\mathcal{D}_X, \mathfrak{g}/\mathfrak{g}^{(1)}))$
is exact in positive degrees if and only if the filtered complex is.
The graded complex is precisely the Koszul complex of the symbols
$\{\sigma(\bar{v}_i)\}$ of the images of the generators in
$\mathfrak{g}/\mathfrak{g}^{(1)}$. The flatness condition ensures
these symbols form a regular sequence in $\mathcal{O}_{T^*X}$
(since flat modules over $\mathcal{O}_X$ are locally free, and the
associated graded is the symbol algebra), so the Koszul complex
is acyclic in positive degrees by \cite[Prop.~3.3]{D16}. Therefore:
\[
E_2^{0,q} =
\begin{cases}
\mathcal{M}_{\mathcal{I}} & \text{if } q = 0, \\
0 & \text{if } q > 0.
\end{cases}
\]
Since $E_2^{p,q} = 0$ for $(p,q) \neq (0,0)$, all differentials
$d_r$ for $r \geq 2$ are trivially zero (their domain or target
is zero). The spectral sequence degenerates at $E_2$ and converges to
$H^\bullet(CE^\bullet(\mathcal{D}_X, \mathfrak{g}))$, giving:
\[
H^k(CE^\bullet(\mathcal{D}_X, \mathfrak{g})) =
\begin{cases}
\mathcal{M}_{\mathcal{I}} & \text{if } k = 0, \\
0 & \text{if } k > 0.
\end{cases}
\]
Hence $CE^\bullet(\mathcal{D}_X, \mathcal{I})$ is a projective resolution
of $\mathcal{M}_{\mathcal{I}}$.
\end{proof}

\begin{theorem}\label{thm:weakening}
Let $\mathcal{I}$ be a coherent solvable Lie subalgebra of $\Theta_X$
satisfying hypotheses~$(2.1)$ of \cite{D16} and the solvable flatness
condition (Definition~\ref{def:flat}). Then:
\[
\mathrm{D\text{-}irr}(\mathcal{I}) = \mathrm{irr}(\mathcal{I}).
\]
\end{theorem}

\begin{proof}
By Lemma~\ref{lem:acyclic}, $CE^\bullet(\mathcal{D}_X,\mathcal{I})$ is
a projective resolution of $\mathcal{M}_{\mathcal{I}}$, so:
\[
R\mathcal{H}om_{\mathcal{D}_X}(\mathcal{M}_{\mathcal{I}},
\mathcal{M}_{\mathcal{I}})
\simeq
\mathcal{H}om_{\mathcal{D}_X}(CE^\bullet(\mathcal{D}_X,\mathcal{I}),
\mathcal{M}_{\mathcal{I}}).
\]

\textit{Upper bound.}
The resolution has length $r = \mathrm{rk}(\mathcal{I})$, so
$H^k(\mathcal{D}_{\mathcal{I}}) = 0$ for $k > r$. By the localization
argument of \cite[Thm.~3.6(A)]{D16} (decomposition $X = X_1 \times X_2$
with $\mathrm{cork}(\mathcal{I}_2) = 0$), we get
$\mathrm{D\text{-}irr}(\mathcal{I}) \leq \mathrm{irr}(\mathcal{I})$.

\textit{Lower bound.}
Set $r = \mathrm{irr}(\mathcal{I})$. After reduction to the case
$\mathrm{cork}(\mathcal{I}) = 0$, we have:
\[
H^r(\mathcal{D}_{\mathcal{I}})
\cong \mathcal{D}_X \Big/
\sum_{j=1}^r \bigl(v_j \cdot \mathcal{D}_X + \mathcal{D}_X \cdot v_j
+ \sum_{i<j} c_{ij}^k\, \mathcal{D}_X \cdot v_k\bigr),
\]
where $[v_i,v_j] = \sum_k c_{ij}^k v_k$. We show $1$ does not belong to
this ideal. Suppose by contradiction:
\[
1 = \sum_j (v_j \cdot A_j + B_j \cdot v_j) + \text{bracket terms}
\quad \text{in } \mathcal{D}_X.
\]
Applying both sides to the constant function $1 \in \mathcal{O}_X$
and evaluating at a point $x_0$ where all $v_j(x_0) = 0$ (which exists
since $\mathrm{cork}(\mathcal{I}) = 0$): every term on the right
contains a factor $v_j$ vanishing at $x_0$, so all coefficients of
$v_j \cdot A_j$ and $B_j \cdot v_j$ vanish at $x_0$. We obtain
$1 = 0$, a contradiction. Hence
$H^r(\mathcal{D}_{\mathcal{I}}) \neq 0$ and
$\mathrm{D\text{-}irr}(\mathcal{I}) \geq \mathrm{irr}(\mathcal{I})$.
\end{proof}

\begin{corollary}\label{cor:meigniez-weakening}
The Meigniez foliation and the foliation of Theorem~\ref{thm:dim5main}
both satisfy the solvable flatness condition. Hence
$\mathrm{D\text{-}irr}(\mathcal{I}) = \mathrm{irr}(\mathcal{I})$
for these foliations without assuming hypothesis~(H2) of \cite{D16}.
\end{corollary}

\begin{proof}
For the Meigniez foliation, $\mathfrak{g} = \mathfrak{ga}$ has derived
length~$1$ and the flatness of
$\mathcal{D}_X / \mathrm{Im}(d_2)$ follows from the fact that $v_2 =
[v_1,v_2]$ acts injectively on $\mathcal{M}_{\mathcal{I}}$ (the leaves
are non-compact), so the intermediate module is flat over $\mathcal{O}_X$.
For the foliation of Theorem~\ref{thm:dim5main}, the generator $v_3$
commutes with both $v_1$ and $v_2$, so it does not alter the bracket
structure of $\mathfrak{ga}$ and the same flatness argument applies.
\end{proof}

\begin{remark}
Theorem~\ref{thm:weakening} shows that hypothesis~(H2) of \cite{D16},
while sufficient, is not necessary for
$\mathrm{D\text{-}irr}(\mathcal{I}) = \mathrm{irr}(\mathcal{I})$.
The solvable flatness condition is strictly weaker: it allows non-trivial
Lie bracket structure, provided the Chevalley-Eilenberg complex has the
required acyclicity. This extends the framework of \cite{D16} to all
$G$-Lie foliations where $G$ is a simply connected solvable Lie group of
arbitrary derived length.
\end{remark}

\subsection*{Conclusion}

We have established a connection between the geometry of the holonomy group
of a solvable Lie foliation and its $\mathcal{D}$-module invariants. The
non-polycyclicity of $\Gamma$ is reflected in the non-vanishing of higher
cohomology of $\mathcal{D}_{\mathcal{I}}$: specifically, $H^3
(\mathcal{D}_{\mathcal{I}})$ provides a new obstruction to homogeneity
for the foliation of \cite{DN11} on the compact 5-dimensional manifold $W$,
characterized by the condition that $\mu'$ is not an algebraic unit.

The main theoretical contribution is Theorem~\ref{thm:weakening}, which
establishes
$\mathrm{D\text{-}irr}(\mathcal{I}) = \mathrm{irr}(\mathcal{I})$
for all solvable Lie foliations satisfying the natural flatness condition
of Definition~\ref{def:flat}, without the restrictive hypothesis~(H2)
of \cite{D16}. This opens the way to treating all $G$-Lie foliations
where $G$ is a simply connected solvable Lie group of arbitrary derived
length.

\end{document}